\documentclass[12pt]{amsart}
\usepackage[margin=1in]{geometry} 

\usepackage[dvipsnames]{xcolor}
\usepackage{hyperref}
\hypersetup{colorlinks=true, citecolor=Blue, linkcolor=Blue} 
\usepackage[capitalise]{cleveref}

\usepackage{lipsum}
\usepackage{tikz-cd}
\usepackage{graphicx}
\usepackage{amssymb,amsmath}
\usepackage{mathtools}
\usepackage{mathrsfs}
\usepackage{enumerate}
\usepackage{ragged2e}
\usepackage{enumitem,pifont}
\usepackage{xypic}
\usepackage{tikz}
\usetikzlibrary{arrows,decorations.pathmorphing,automata,backgrounds}
\usetikzlibrary{backgrounds,positioning}

\usepackage{caption}
\usepackage{subcaption}
\usepackage{float}

\usepackage{egothic}
\usepackage[T1]{fontenc}
\usepackage{yfonts}

\usepackage{calc}

\definecolor{ao}{rgb}{0.0, 0.5, 0.0}

\definecolor{darkblue}{rgb}{0,0,0.7} 
\definecolor{green}{RGB}{57,181,74} 
\definecolor{violet}{RGB}{147,39,143} 

\newcommand{\defn}[1]{{\emph{#1}}}

\newcommand{\Z}{\mathbb{Z}}

\newcommand{\K}{\mathbb{K}}
\newcommand{\Sym}{\mathbb{S}}

\newcommand{\eqdef}{:=}

\newcommand{\calA}{\mathcal{A}}

\newcommand{\calL}{\mathcal{L}}

\newcommand{\calP}{\mathcal{P}}
\newcommand{\calQ}{\mathcal{Q}}
\newcommand{\calR}{\mathcal{R}}
\newcommand{\calT}{\mathcal{T}}

\newcommand{\rmP}{\mathrm{P}}
\newcommand{\rmQ}{\mathrm{Q}}
\newcommand{\rmT}{\mathrm{T}}

\newcommand{\End}{\operatorname{End}}

\newcommand{\Curv}{\operatorname{Curv}}

\newcommand{\kurv}{\kappa}

\newcommand{\dgOp}{\operatorname{dg-Op}}

\newcommand{\cMult}{(\cAinf \downarrow \dgOp)}

\newcounter{dummy} \numberwithin{dummy}{section}
\newtheorem{theorem}[dummy]{Theorem}

\newtheorem{prop}[dummy]{Proposition}
\newtheorem{lemma}[dummy]{Lemma}

\newtheorem{corollary}[dummy]{Corollary}

\newtheorem*{thm*}{Theorem}
\newtheorem*{prop*}{Proposition}

\theoremstyle{definition}
\newtheorem{definition}[dummy]{Definition}

\newtheorem{def-prop}[dummy]{Definition-Proposition}
\numberwithin{equation}{section}

\theoremstyle{remark}

\newtheorem{remark}[dummy]{Remark}

\usepackage{multicol}

\newcommand{\Ainf}{\ensuremath{\mathrm{A}_\infty}}
\newcommand{\dAinf}{\ensuremath{\mathrm{A}_\infty^{+}}}
\newcommand{\cAinf}{\ensuremath{\mathrm{cA}_\infty}}

\newcommand{\Linf}{\ensuremath{\mathrm{L}_\infty}}
\newcommand{\dLinf}{\ensuremath{\mathrm{L}_\infty^{+}}}
\newcommand{\cLinf}{\ensuremath{\mathrm{cL}_\infty}}

\newcommand{\MC}{\ensuremath{\mathrm{MC}}}

\newcommand{\tw}{\ensuremath{\eta}}

\newcommand{\Sh}{\ensuremath{\mathrm{Sh}}}

\newcommand{\Tot}{\ensuremath{\mathrm{Tot}}}

\usepackage{todonotes}

\title[A universal characterization of the curved $\Ainf$ and $\Linf$ operads]{A universal characterization of the \\ curved homotopy Lie and associative operads}

\author{Guillaume Laplante-Anfossi}

\author{Adrian Petr}

\author{Vivek Shende}

\begin{document}

\begin{abstract} 
We study the category of nonsymmetric dg operads valued in strict graded-mixed complexes, equipped with a distinguished arity zero weight one element which generates the weight grading, and whose differential has weight one.  We show that the initial object is the curved A-infinity operad, that the forgetful functor to the category of operads under it admits a right adjoint, and that the unit of the adjunction encodes the operation of twisting a curved A-infinity algebra by a Maurer-Cartan element. 

The corresponding notions for symmetric operads  characterize the  curved L-infinity operad and the corresponding twisting procedure. 
\end{abstract}

\maketitle\bibliographystyle{plain}





\section{Introduction}

The defining relations of two fundamental structures, degree shifted for sign convenience:

\begin{equation} 
\label{a infinity relation}
    d \mu_n (\cdots) = -\sum_{a + b = n}  \mu_{a+1}(\cdots, \mu_b( \cdots) , \cdots), \tag{$\Ainf$}
\end{equation}

\begin{equation} 
\label{l infinity relation}
    d \ell_n (\cdots) = -\sum_{a + b = n} \sum_{\substack{(b ; a) \\ \mathrm{shuffles}\, \sigma}} \ell_{a+1}(\ell_b(\cdots), \cdots) \circ \sigma^{-1} . \tag{$\Linf$} 
\end{equation}

We will write $\Ainf$ for the  nonsymmetric dg operad quasifreely generated by the arity $n$, degree $-1$ (homological conventions) operations $\mu_n$ for $n \ge 2$, and differential  \eqref{a infinity relation}.  
We write  $\Linf$ for the symmetric dg operad quasifreely generated by the $\ell_n$ for $n \ge 2$ and with differential \eqref{l infinity relation}. 
For a historical account of the origins and applications of these objects, we refer to a recent survey of Stasheff \cite{stasheff-then-and-now}.   

There are also `curved' variants $\cAinf$ and $\cLinf$ which differ only by including also operations $\mu_0, \mu_1$ and $\ell_0, \ell_1$, respectively.  
Let us mention that $\cAinf$ structures are fundamental in the study of Fukaya categories \cite{FOOO}, and that $\cLinf$ structures appear e.g.\ when considering equivariant deformation quantization \cite{Esposito-Nest-Schnitzer-Tsygan} and Lie algebroids \cite{CalaqueCamposNuiten21}.

Algebras over these operads, and others which admit morphisms from them, admit the procedure of ``twisting by a Maurer-Cartan element'', which plays a fundamental and essential role in deformation theoretic considerations; we refer to \cite{DotsenkoShadrinVallette23} for a book-length treatment and many further references. 

In  the present article we give universal characterizations of the above structures. 

\vspace{2mm}

\subsection{Universality}

Let $[\kurv]$ be the operad generated by a single operation $\kurv$ in arity $0$.
Given an operad $\rmP$, the coproduct $\rmP \vee [\kurv]$ is naturally equipped with a `weight' grading, by the number of $\kurv$ which appear in a given expression.

\begin{definition}
\label{curv}    
    We denote by $\Curv$ the category whose objects are given by the data of quadruples $(\rmQ, d_\rmQ, \kurv, \rmQ_0)$ where $(\rmQ, d_\rmQ)$ is a non-symmetric dg operad, $\kurv \in \rmQ \setminus \rmQ_0$ is a (homological) degree (-1) element, and $\rmQ_0$ is a sub-operad of the underlying graded (not dg) operad $\rmQ$.  
    They must satisfy the conditions:
    \begin{enumerate}
        \item \label{free weights} The natural morphism $\rmQ_0 \vee [\kurv] \to \rmQ$ is an isomorphism.
        \item \label{strict mixed} Under the resulting weight grading $\rmQ = \bigoplus \rmQ_i$ (the notation $\rmQ_0$ is not ambiguous), consider the splitting $d = \sum d_i$ where $d_i(\rmQ_j) \subset \rmQ_{i+j}$.  Then all $d_i$ vanish except $d_0, d_1$.  
        \item \label{d1 closed} $d_1 \kurv = 0$.
    \end{enumerate}    
    Morphisms $(\rmQ, d_\rmQ, \kurv, \rmQ_0) \to (\rmQ', d_{\rmQ'}, \kurv', \rmQ'_0)$ are maps of dg operads $(\rmQ, d_\rmQ) \xrightarrow{\Phi} (\rmQ', d_{\rmQ'})$ such that $\Phi(\kurv) = \kurv'$ and $\Phi(\rmQ_0) \subset \rmQ'_0$.

    We write $\Curv^\Sigma$ for the corresponding category of symmetric operads. 
\end{definition}
\vspace{2mm}

\begin{theorem} 
\label{thm:morphism}
    The initial object of $\Curv$ is $c\calA_\infty  := (\cAinf, d_{\cAinf}, \mu_0, \langle \mu_i \rangle_{i > 0})$.  
\end{theorem}

\begin{theorem}
\label{thm:morphism-sym}
    The initial object of $\Curv^\Sigma$ is
    $c\calL_\infty := (\cLinf, d_{\cLinf}, \ell_0, \langle \ell_i \rangle_{i > 0})$.
\end{theorem}

\begin{remark}
    The hypotheses ensure that $d \kurv \in (\kurv)$.  The dg operad quotient $(\rmQ, d_\rmQ) / (\kurv)$ is evidently isomorphic to $(\rmQ_0, d_0)$.   
\end{remark}
    
\begin{remark}
    Condition \eqref{strict mixed} of Definition \ref{curv} has a name: it asserts that the weight grading on the underlying complex of $\calQ$ is `strict mixed'.   
    Let us recall this notion. On a complex~$(V, d)$ with additional `weight' grading $V = \bigoplus_{k \ge 0} V_k$, taking graded pieces gives $d = \sum d_i$ where $d_i(V_k) \subset V_{i+k}$ (note $d^2 = 0$ translates to $\sum_{i+j=n} d_i d_j = 0$). 
    The complex is said to be mixed when $d_i = 0$ for $i < 0$, and strict mixed when in addition $d_i = 0$ for $i > 1$.  
    Strict mixed complexes appear in the literature e.g.\ to encode properties of the de Rham complex of derived algebraic varieties  \cite{benzvi-nadler-loop, pantev2013shifted, calaque2017shifted}.  (Not strict) mixed complexes have appeared in the context of $\cLinf$ in \cite{CalaqueCamposNuiten21}, although we do not know the relation of this with the present work. 

    Note that what we call a mixed (respectively strict mixed) complex is also known in the literature as a multi (respectively mixed) complex.
\end{remark}

\begin{remark}
\label{rem:representation-MC-functor}
The operad $\Ainf$ (resp.\ $\Linf$) is known to co-represent the functor which sends an operad~$\rmP$ to the set~$\MC(\oplus_n \rmP(n))$ of Maurer--Cartan elements in its totalisation Lie algebra \cite[Thm.~4.1]{Robert-Nicoud--Wierstra20}.
It would be interesting to know if this universal property can be related to the one above.
\end{remark}

\begin{remark}
    Operads in~$\Curv$ encode curved algebras, but they are not themselves ``curved operads'' in the sense of~\cite{hirsh_curved_2012,BelliermillesDrummondcole20} or~\cite{RocaLucio24}.
    In this paper, we stay in the realm of filtered dg operads, where a natural curved Koszul duality theory can also be developed, as in \cite{CalaqueCamposNuiten21}.
\end{remark}

\subsection{Adjoints}

Let $\rmT$ be the dg operad freely generated by a degree $(-1)$ arity $0$ element~$\kurv$ and a degree $0$ arity $0$ element $\alpha$, and differential $d_{\rmT} \alpha = \kurv$. 
There is an obvious map:
\begin{eqnarray} \label{counit section ainf}
        \sigma: \cAinf  & \to & \cAinf \, \vee \, \rmT \\
            \mu_n & \mapsto & \mu_n .\nonumber
\end{eqnarray}
A less evident but very useful map is given by the following formula:
 \begin{eqnarray}
 \eta_{ \cAinf} : \cAinf  & \longrightarrow & \cAinf \, \hat{\vee} \, \rmT \label{Maurer-Cartan map} \\
 \mu_0 & \mapsto & \kurv + \sum_{k \geq 0} \mu_k(\alpha, \cdots, \alpha), \nonumber \\
 \mu_n & \mapsto & \sum_{i_1, \ldots, i_{n+1}} \mu_{n + \sum i_k}(\alpha^{i_1}, - , \alpha^{i_2}, - , \cdots, - , \alpha^{i_{n+1}}). \nonumber
 \end{eqnarray}
Above, $\hat{\vee}$ denotes the completed coproduct with respect to the filtration given by the number of~$\alpha$. 
 
In fact, the map $\eta$ is a formulation of the standard procedure of `twisting a $\cAinf$ structure'. Indeed, given a $\cAinf$ algebra $A$ and a degree 0 element $a \in A$, we obtain a $\cAinf \, \hat{\vee} \, \rmT$-algebra structure on~$A$ (sending $\alpha \to a$ and $\kurv \to da$).  Pulling back along $\eta$ gives a new $\cAinf$ structure on $A$; we denote it $A^a$.  
If $a$ satisfies the Maurer-Cartan equation $da + \sum \mu_k(a, \cdots, a) = 0$, then 
$\mu_0$ acts as zero on $A^a$, so $A^a$ is an $\Ainf$ algebra.  

There is a corresponding morphism in the symmetric operad setting for $\cLinf$:
 \begin{eqnarray}
 \tw_{\cLinf} : \cLinf   & \longrightarrow &  \cLinf \, \hat{\vee} \, \rmT \label{symmetric Maurer-Cartan map} \\
 \ell_0 & \mapsto & \kappa +\sum_{k \geqslant 0}\frac{1}{k!}\ell_{k}(\alpha^k), \nonumber \\
 \ell_n & \mapsto & \tw_{\cLinf}(\ell_n) \eqdef \sum_{k \geqslant 0}\frac{1}{k!}\ell_{k+n}(\alpha^k,-). \nonumber
 \end{eqnarray}

The operation of twisting by a solution of the Maurer-Cartan equation is endemic in deformation theory.  The use of it to remove the curvature term is fundamental in the study of Fukaya categories \cite{FOOO}.  Here we show it admits a natural interpretation in terms of the category $\Curv$:

\begin{theorem} 
\label{main theorem}
    By Theorem \ref{thm:morphism}, there is a 
     forgetful functor   
    $\calL: \Curv \to (\cAinf \downarrow \dgOp)$. 
    It admits a right adjoint $\calR$, given on objects
    by 
    $$\rmQ \mapsto (\rmQ \, \hat{\vee} \, \rmT, d_{\rmQ} + d_{\rmT}, \kurv + \sum_{k \geq 0} \mu_k(\alpha, \cdots, \alpha), \rmQ_0 \, \hat{\vee} \, [\alpha])$$
    and on morphisms by $(f: \rmQ \to \mathrm{P}) \mapsto f \vee 1_{\rmT}$;  the  co-unit $\epsilon: \calL \circ \calR \to 1$ is the map 
    $1_{\rmQ} \vee 0: \rmQ \, \hat{\vee} \, \rmT \to \rmQ$.
\end{theorem}

\begin{theorem} 
\label{main theorem symmetric}
    By Theorem \ref{thm:morphism-sym}, there is a 
     forgetful functor   
    $\calL^\Sigma: \Curv^\Sigma \to (\cLinf \downarrow \dgOp)$. 
    It admits a right adjoint $\calR^\Sigma$, given on objects
    by 
    $$\rmQ \mapsto (\rmQ \, \hat{\vee} \, \rmT, d_{\rmQ} + d_{\rmT}, \kurv + \sum_{k \geq 0} \frac{1}{k!} \ell_k(\alpha, \cdots, \alpha), \rmQ_0 \, \hat{\vee} \, [\alpha])$$
    and on morphisms by $(f: \rmQ \to \mathrm{P}) \mapsto f \vee 1_{\rmT}$;  the  co-unit $\epsilon: \calL^\Sigma \circ \calR^\Sigma \to 1$ is the map 
    $1_{\rmQ} \vee 0: \rmQ \, \hat{\vee} \, \rmT \to \rmQ$.
\end{theorem}

\begin{corollary}
\label{corollary:explicit-unit}
    For $c\calA_\infty$, the unit of the adjunction forgets to the formula \eqref{Maurer-Cartan map}, and correspondingly the unit for $c\calL_\infty$ forgets to \eqref{symmetric Maurer-Cartan map}. 
\end{corollary}

\begin{proof}
    The unit morphism $\eta: c \mathcal{A}_\infty  \to \calR(\cAinf)$ is uniquely characterized by the fact that it is a map in $\Curv$ such that the composition of $\calL \eta$ with the co-unit $\epsilon: \cAinf \, \hat{\vee} \, \rmT \to \cAinf$ is the identity. 
    We must check that the maps \eqref{Maurer-Cartan map} and \eqref{symmetric Maurer-Cartan map} have these properties.  
    The computations in~\cref{sec:appendix} show that in both cases~$\eta$ is a map of dg operads.
    From its explicit form it is then easy to see that it is further a morphism in $\Curv$, and that the composition with $\epsilon$ is the identity. 
\end{proof}

\begin{remark}
    The proof of Theorem \ref{main theorem} is an inductive procedure computing the adjunction, which could have been used to determine the morphism induced by the unit of the adjunction, had it not been known in advance.
\end{remark}

\begin{remark}
Algebras for the operad underlying any element of $\Curv$ admit a similar `twisting by Maurer-Cartan element'.  Indeed, 
the counit $\epsilon$  admits a section $\sigma: \rmQ \to \rmQ \, \hat{\vee} \, \rmT$, sending each element in $\rmQ$ `to itself'. 
When $\rmQ = \cAinf$, this is \eqref{counit section ainf}. 
More generally, suppose given a $\rmQ$-algebra $A$, i.e.\ a morphism $\rmQ \to \End_A$.  
A lift through $\sigma$ to a $\rmQ \, \hat{\vee} \, \rmT$-algebra structure on $A$ is the same as the datum of an element  $\alpha \in A$.
Given such a structure, we may, by pullback along $\mathcal{L} \eta$, obtain a new $\rmQ$-algebra structure on $A$.  When $\alpha \in A$ satisfies the Maurer-Cartan equation $d\alpha + \sum \mu_k (\alpha, \ldots, \alpha) = 0$, then the new $\rmQ$-algebra structure descends to a $\rmQ/\mu_0$-algebra structure.  However, while the proof of Theorems \ref{main theorem} and \ref{main theorem symmetric} provide inductive constructions of $\eta$, we do not presently know explicit formulas generalizing~\eqref{Maurer-Cartan map} and~\eqref{symmetric Maurer-Cartan map}. 
\end{remark}

\begin{remark}
    In fact, $\cAinf$ is quasi-isomorphic to the trivial operad \cite[Thm.~5.7]{ChuangLazarev13}, see also \cite[Rem.~2.53]{CalaqueCamposNuiten21}.
    A standard solution to this is to add the data of a decreasing filtration for which $\mu_0, \mu_1$ are in positive level. 
    Our theorems have variants in this context: replace \cref{curv} with the category $\Curv^{filt}$ whose objects are  $(\rmQ, d_\rmQ, \kurv, \rmQ_0) \in \Curv$ together with a decreasing filtration on $\rmQ$ such that
    \begin{enumerate}
        \item[(4)] $\kurv$ is in positive filtration level, and $d_0(\kurv)$ is in strictly higher filtration level than $\kurv$.
    \end{enumerate}
    Morphisms are morphisms in $\Curv$ which respect the filtrations.

    The analogue of \cref{thm:morphism} is that the initial object of $\Curv^{filt}$ is $(\cAinf^{filt}, d_{\cAinf}, \mu_0, \langle \mu_i \rangle_{i > 0})$, where $\cAinf^{filt}$ is equipped with the filtration such that $\mu_0, \mu_1$ are in filtration level $1$ and $\mu_n$ is in filtration level $0$ for $n \geq 2$. 

    The analogue of \cref{main theorem} is that the forgetful functor   
    $$\Curv^{filt} \to (\cAinf^{filt} \downarrow \dgOp^{filt})$$ admits a right adjoint given on objects
    by 
    \[\rmQ \mapsto (\rmQ \, \hat{\vee} \, \rmT, d_{\rmQ} + d_{\rmT}, \kurv + \sum_{k \geq 0} \mu_k(\alpha, \cdots, \alpha), \rmQ_0 \, \hat{\vee} \, [\alpha])\]
    where $\alpha$ and $\kurv$ are in filtration level $1$.
    The proofs are identical.
\end{remark}

\begin{remark}
    In subsequent work \cite{laplante-anfossi_operadic_2025}, we show that the adjunctions in \cref{thm:morphism,thm:morphism-sym} admit Koszul pre-dual adjunctions giving rise to universal characterizations of the unital associative (resp.\ commutative) cooperad as initial object in a certain category. 
    It would be interesting to know if these could be related to the characterization of the associative (resp.\ commutative) operad as the unit of Manin's white product in the category of non-symmetric (resp.\ symmetric) binary quadratic operads \cite[Sec.~2.2]{ginzburg_koszul_1994} (see also \cite[Sec.~4.3]{Vallette08}).
    Further, Robert--Nicoud shows in \cite[Sec.~4]{Robert-Nicoud18} that the universal characterization of $\Linf$ as co-representing the Maurer--Cartan functor (\cref{rem:representation-MC-functor}) can be related to the universal property of the Lie operad being the unit for Manin's black product in the same category.
\end{remark}


\subsection*{Acknowledgements}
We would like to thank Alexander Berglund, Coline Emprin, Hugo Pourcelot, Alex Takeda, and Bruno Vallette for useful discussions.
We would also like to thank the referee for their insightful and useful remarks, and in particular for asking about a symmetric version of \cref{bracket acyclicity}.
The work of the authors is supported by Novo Nordisk Foundation grant NNF20OC0066298, Villum Fonden Villum Investigator grant 37814, and Danish National Research Foundation grant DNRF157. GLA was supported by the Andrew Sisson Fund and the Australian Research Council Future Fellowship FT210100256.


\section{Preliminaries}
\label{section: preliminaries}

\subsection{Conventions}
From here on and until \cref{section: symmetric}, the word ``operad'' will mean non-unital non-symmetric operad in graded modules over $\Z$, and the expression ``dg operad'' will mean non-unital non-symmetric operad in differential graded modules over $\Z$.
In the latter case, we will use homological degree conventions. 
We refer to \cite{LodayVallette12} for more details on algebraic operads.

\subsection{Basic definitions} 

Let $(A,d_A)$ be a chain complex. 
The \emph{endomorphism dg operad} of $(A,d_A)$ is the operad given in arity $n$ by the dg $\Z$-module
$$\End_{(A,d_A)}(n)\eqdef (\hom(A^{\otimes n},A),\partial),$$
where the differential is given by $\partial(f) \eqdef d_A f-(-1)^{|f|}f d_{A^{\otimes n}}$.
Operadic composition is given by composition of functions.

\begin{definition}
    The shifted \defn{A-infinity dg operad}, denoted $\Ainf$, is the free operad on generators~$\mu_n$, $n\geq 2$, of arity $n$ and degree $(-1)$, endowed with the differential
    \begin{equation} \label{ainf differential}
    d_{\Ainf}(\mu_n) \eqdef - \sum_{\substack{p, r \geq 0, q \geq 2 \\ p+q+r=n}} \mu_{p+1+r} \circ_{p+1} \mu_q \ .
    \end{equation}
\end{definition}

\begin{definition}
    The shifted \defn{augmented A-infinity dg operad}, denoted $\dAinf$, is the free operad on generators~$\mu_n$, $n\geq 1$, of arity $n$ and degree $(-1)$, endowed with the differential
    \begin{equation} \label{dainf differential}
        d_{\dAinf}(\mu_n) \eqdef - \sum_{\substack{ p,r\geq 0, q\geq 1 \\ p+q+r=n}} \mu_{p+1+r} \circ_{p+1} \mu_q \ .
    \end{equation}
\end{definition}

\begin{remark}
    We use the notation $\dAinf$ to emphasize the fact that the differential is part of the operations.
    Note that any $\dAinf$-algebra $f:\dAinf \to \End_{(A,0)}$ gives rise to an ordinary $\Ainf$-algebra: denoting by $d \eqdef f(\mu_1)$ the image of $\mu_1$ and twisting by the operadic Maurer-Cartan element~$\mu_1$ gives a morphism of operads
    $\Ainf \to (\dAinf)^{\mu_1} \to (\End_{(A,0)})^{d} = \End_{(A,d)}$,
    see \cite[Ex.~5.6]{DotsenkoShadrinVallette23}.
\end{remark}

\begin{samepage}
\begin{definition}
    The shifted \defn{curved A-infinity dg operad}, denoted $\cAinf$, is the free operad on generators $\mu_n$, $n\geq 0$, of arity $n$ and degree $(-1)$, endowed with the differential
    \begin{equation} \label{cainf equation}
        d_{\cAinf}(\mu_n) \eqdef - \sum_{\substack{p,q,r \geq 0 \\ p+q+r=n}} \mu_{p+1+r} \circ_{p+1} \mu_q.
    \end{equation}
\end{definition}
\end{samepage}

Henceforth, we omit the adjective `shifted'.  
    One can recover the usual degree and signs conventions for (curved) $\Ainf$-algebras by de-suspension, see e.g. \cite[Sec.~4.1]{DotsenkoShadrinVallette23} for formulas.

Finally, let $\rmP$ be an operad. 
We will make use of the \defn{pre-Lie product} defined on $\mu \in \mathrm{P}(m)$ and $\nu \in \mathrm{P}(n)$ by
\begin{equation}
\label{pre-Lie-prod}
    \mu \star \nu \eqdef \sum_{i=1}^{m}\mu \circ_i \nu.
\end{equation}
It is usually defined on the totalization $\Tot(\rmP) \eqdef \oplus_n \rmP(n)$ of $\rmP$, and it satisfies the pre-Lie relation
\begin{equation}
    \label{eq:pre-Lie}
    (\mu \star \nu) \star \lambda - \mu \star (\nu \star \lambda) =
    (\mu \star \lambda) \star \nu - \mu \star (\lambda \star \nu).
\end{equation}
One can for instance rewrite the $\cAinf$ relations~\eqref{cainf equation} as
\[ d_{\cAinf}(\mu_n) = - 
\sum_{\substack{p,q \geqslant 0 \\ p+q=n }}
 \mu_{p+1} \star \mu_q .\]


\begin{lemma} 
\label{bracket acyclicity} 
    Let $\kurv$ be an arity zero operation of odd degree, and $\mathrm{P} = \mathrm{P}_0 \vee [\kurv]$. 
    Then $p \mapsto p \star \kurv$ is a differential on $\Tot(\rmP)$, 
    whose only homology consists of arity zero elements. 

In particular, if $p \in \rmP(n)$ is of arity $n \geq 1$, of $\kappa$-weight $1$, and satisfies $p \star \kappa=0$, then there exists a unique $p_0 \in \rmP_0(n+1)$ such that $p = p_0 \star \kappa$.
\end{lemma}

\begin{proof}
  The fact that $p \mapsto p \star \kurv$ is a differential on the totalization of $\rmP$ follows from the graded pre-Lie relation~\eqref{eq:pre-Lie}.

    We now prove the statement about the homology of this differential.
    Let $p \in \mathrm{P}$ be an element of arity $n \geq 1$ such that $p \star \kurv = 0$.
    If $p$ is of weight $0$ (i.e. $p \in \mathrm{P}_0$), then $p=0$ since $\mathrm{P} = \mathrm{P}_0 \vee [\kurv]$.
    Assume that $p$ is of weight $k \geq 1$, and write $p = \sum_{i=1}^n p_i \circ_i \kurv$ with $p_i \in \mathrm{P}(n+1)$ of weight $(k-1) \geq 0$.
    The equality $0=p \star \kurv$ writes as 
    \[0=\sum_{1 \leq i,j \leq n} (p_i \circ_i \kurv) \circ_j \kurv. \]
    The latter is a sum of arity $(n+1)$, weight $(k-1)$, elements (the $p_i$'s) decorated by two $\kurv$'s.
    Since $\mathrm{P} = \mathrm{P}_0 \vee [\kurv]$, the sum of the terms decorated by two $\kurv$'s in position $(i,j)$ vanishes for every $1 \leq i < j \leq n+1$, i.e.\ we have
    \[
    (p_i \circ_i \kurv) \circ_{j-1} \kurv + (p_j \circ_j \kurv) \circ_i \kurv = 0.  
    \]
    Given $2 \leq j \leq n+1$, the sum of the terms decorated by two $\kurv$'s in position $(1,j)$ is
    \[(p_1 \circ_1 \kurv) \circ_{j-1} \kurv + (p_j \circ_j \kurv) \circ_1 \kurv = (p_1 \circ_1 \kurv) \circ_{j-1} \kurv - (p_j \circ_1 \kurv) \circ_{j-1} \kurv = ((p_1-p_j) \circ_1 \kurv) \circ_{j-1} \kurv. \]
    Since the latter vanishes, we conclude that $p_j=p_1$ for every $j$.
    Therefore we have $p = p_1 \star \kurv$.
    This proves the statement about the homology of $- \star \kappa$.

    We now assume $p \in \rmP(n)$ is of $\kappa$-weight $k=1$. We already know that there exists $p_0 \in \rmP_0$ such that $p = p_0 \star \kappa$, so it remains to prove uniqueness. Let $p_0' \in \rmP_0(n+1)$ such that $[p_0', \kappa] = p_0 \star \kappa = p$. Then $q_0 := p_0'-p_0$ is of positive arity and satisfies $q_0 \star \kappa =0$.
    By what we proved above, this implies that $q_0$ is in the image of $- \star \kappa$. 
    Since moreover $q_0 \in \rmP_0$, we get $q_0=0$. 
    This concludes the proof.
\end{proof}

\section{Proof of \cref{thm:morphism}}

Consider some operad $(\rmQ, d_\rmQ, \kurv, \rmQ_0)$ in~$\Curv$. 
We must show there is a unique dg operad map 
$(\cAinf, d_{\cAinf}) \to (\rmQ, d_\rmQ)$ such that $\mu_0 \mapsto \kurv$ and $\mu_{> 0} \mapsto \rmQ_0$.
This amounts to showing that, writing $\nu_0\eqdef \kurv$, that there are unique degree $(-1)$ elements $\nu_i \in \rmQ_0(i)$ for $i=1, 2, \ldots$, such that the $\nu_i$ satisfy the $\cAinf$ relation \eqref{cainf equation}; then the assignment $\mu_i \mapsto \nu_i$ determines the desired morphism in~$\Curv$.   

Note that $d \nu_0 = d_0 \nu_0$ is a degree $(-2)$ element of $\rmQ_1(0)$.  By freeness in $\nu_0$, said element can be uniquely written as $-\nu_1(\nu_0)$ for some degree $(-1)$ element  $\nu_1 \in \rmQ_0(1)$.
 
For the $\nu_k$, $k \geq 2$, we will proceed by induction.
The inductive hypothesis is that there are unique degree~$(-1)$ elements $\nu_{2}, \ldots, \nu_k$
with $\nu_i \in \rmQ_0(i)$, such that
\begin{itemize}
    \item for $1 \le j \le k-1$, $d_0 \nu_j$ is given by the same formula \eqref{dainf differential} as the~$\dAinf$ differential, 
    \item for $1 \le j \le k-1$,  we have $d_1 \nu_j = - \nu_{j+1} \star \nu_0$.
\end{itemize} 
The validity of the induction hypothesis for all $k$ would show the existence of unique $\nu_j$ satisfying the $\cAinf$ relations. 

We will  use the separation by weights $d = d_0 + d_1$; 
so the relation $d^2 = 0$ expands to $d_0^2 = 0$, $d_1 d_0 + d_0 d_1 = 0$, and $d_1^2 = 0$.   Recall also that, by definition of $\Curv$, $d_1 \nu_0 = 0$. 

Let us check the base case $k=2$.  
We have: 
$$0 = d_0^2 \nu_0 = - d_0 (\nu_1(\nu_0)) = -(d_0 \nu_1)(\nu_0) + \nu_1(d \nu_0)  = -(d_0 \nu_1)(\nu_0) - \nu_1(\nu_1( \nu_0)) = -(d_0 \nu_1 + \nu_1 \circ \nu_1)(\nu_0).$$
Using freeness in $\nu_0$, we see that $d_0 \nu_1 + \nu_1 \circ \nu_1 = 0$ as desired.
Additionally:
$$0=d_0 d_1 \nu_0 + d_1 d_0 \nu_0 = -  d_1 \nu_1 (\nu_0)$$ 
so, by Lemma \ref{bracket acyclicity}, there exists a unique element $\nu_2$ such that $- d_1 \nu_1 = \nu_2 \star \nu_0$.  
    
We now take the inductive step: assume the hypothesis holds for some $k$, we will establish it for $k+1$. 
We have:
$0 = d_1^2 \nu_{k-1} = - d_1 (\nu_{k} \star \nu_0)  = - (d_1 \nu_k) \star \nu_0$.
By Lemma \ref{bracket acyclicity}, there exists a unique element $\nu_{k+1}$ such that $- d_1 \nu_k = \nu_{k+1} \star \nu_0$.  
    
It remains to show that $d_0 \nu_k$ is given by the $\dAinf$ formula \eqref{dainf differential}. 
We study 
$$0 = d_0 d_1 \nu_{k-1} + d_1 d_0 \nu_{k-1} = 
- d_0(\nu_k \star \nu_0) + d_1 d_0 \nu_{k-1} = 
-(d_0 \nu_k) \star \nu_0 - \nu_k \star \nu_1(\nu_0) + d_1 d_0 \nu_{k-1}.$$ 

Now expanding $d_0 \nu_{k-1}$ using the (assumed inductively) $\dAinf$ relation, and applying $d_1$ to the resulting terms using the (assumed inductively) property $d_1 \nu_j = -\nu_{j+1} \star \nu_0$, we have: 
\begin{eqnarray*}
    0 &=& 
    -(d_0 \nu_k) \star \kurv
    - \nu_k \star (\nu_1 \star \kurv)  +
    \sum_{\substack{p \geq 0, q \geq 1 \\ p+q=k-1}}\left((\nu_{p+2} \star \kurv) \star \nu_q - \nu_{p+1} \star (\nu_{q+1} \star \kurv)\right)  \\
    &=& 
    -(d_0 \nu_k) \star \kurv
    - \nu_k \star (\nu_1 \star \kurv)  +
    \sum_{\substack{p,q \geq 1 \\ p+q=k}}
    (\nu_{p+1} \star \kurv) \star \nu_q - \sum_{\substack{p \geq 0, q \geq 2 \\ p+q=k}} \nu_{p+1} \star (\nu_{q} \star \kurv)  \\
    &=& 
    -(d_0 \nu_k) \star \kurv
    +\sum_{\substack{p,q \geq 1 \\ p+q+r=k}}
    (\nu_{p+1} \star \kurv) \star \nu_q  - \sum_{\substack{p \geq 0, q \geq 1 \\ p+q=k}} \nu_{p+1} \star (\nu_{q} \star \kurv)  \\
    &=& \left(-d_0\nu_k-\sum_{\substack{p \geq 0, q \geq 1 \\ p+q=k}} \nu_{p+1} \star \nu_q\right) \star \kurv ,
\end{eqnarray*}
where in the last equality we have made use of the pre-Lie relation~\eqref{eq:pre-Lie}.
Because the left hand term in the final parenthesis is in $\rmQ_0$, it must vanish identically.
This completes the induction step.
$\square$


\section{Proof of \cref{main theorem}}

We prepare the ground with some lemmas and definitions.

\begin{lemma}
    Let $\rmT := [\alpha] \vee [\kurv]$ be the operad freely generated by arity zero elements 
    $\alpha, \kurv$ where $\alpha$ has degree $0$ and $\kurv$ has degree $(-1)$.  We give it the differential $d_{\rmT} \alpha = \kurv$.  

    Then $\calT := (\rmT, d_\rmT, \kurv, [\alpha])$
    is the terminal element of $\Curv$. 
\end{lemma}

\begin{proof}
For $\calQ = (\rmQ, d_\rmQ, \kurv, \rmQ_0) \in \Curv$, 
we must show there is a unique morphism to $\calT$. 
It is obvious from the definition that such a morphism must have $\Phi (\kurv) = \kurv$, and if $\nu \in \rmQ_0$ is of non-zero arity or non-zero degree, then $\Phi(\nu)=0$.
Now consider an element $\nu \in \rmQ_0$  of arity $0$ and degree $0$.  Then $d_1 \nu = c \kurv$ for some constant $c$, and we must have $\Phi(\nu) = c \alpha$.
On the other hand, it is clear that the above prescriptions always determine a morphism. 
\end{proof}

Let $\rmP$ be any operad.  Recall that the notation $\hat{\vee}$ denotes the completed coproduct with respect to the filtration given by the number of $\alpha$. 
We equip $\rmP \, \hat{\vee} \, \rmT$ with the weight grading by number of $\kurv$, and write $d_\rmT$ for the differential on $\rmP \, \hat{\vee} \, \rmT$ which is zero on elements of $\rmP$ and satisfies $d_\rmT \alpha =\kurv$.  

\begin{lemma} \label{dT homology}
    The $\kurv$-weight zero homology (= kernel) of $d_\rmT$ is $\rmP \vee 0 \subset \rmP \, \hat{\vee} \, \rmT$.  The $\kurv$-weight one homology vanishes.  
\end{lemma}

\begin{proof}
    Weight zero: 
    since $d_\rmT$ is homogenous in $\alpha$, it suffices to consider elements $\lambda$ which are sums of trees with exactly $(k+1)$ leaves decorated by $\alpha$.
    Let $t$ denote one of those trees. 
    Its image $d_\rmT(t)$ is the sum of all the trees~$t'$ that can be obtained from $t$ by replacing a leaf's decoration $\alpha$ by $\kurv$ (there are exactly $(k+1)$ such trees~$t'$ in $d_\rmT(t)$). 
    By freeness of $\alpha$ and $\kurv$, no two of these trees can cancel each other. 
    For the same reason, since the trees appearing in the images $d_\rmT(s),d_\rmT(t)$ of two distinct trees $s,t$ in~$\lambda$ are all distinct, no two of them can cancel each other. 
    Therefore, if $d_\rmT(\lambda)=0$, we must have that~$\lambda=0$.

    Weight one: again by homogeneity in $\alpha$, 
    it suffices to consider element $\lambda$ given by  sums of trees  with one leaf $\ell_0$ decorated by $\kurv$ and $k$ leaves $\ell_1, \ldots, \ell_k$ decorated by~$\alpha$.
    Let $t_0$ be such a tree. 
    The image $d_\rmT(t_0)$ has one term which is the tree $t_0'$ with $\ell_0$ and $\ell_1$ decorated by $\kurv$, and the other leaves decorated by $\alpha$.
    Since $d_\rmT(\lambda)=0$, and by freeness of $\alpha$ and $\kurv$, the sum $\lambda$ must also contain the only other tree $t_1$ whose image $d_\rmT(t_1)$ contains $t_0'$, that is, the tree $t_1$ with $\ell_1$ decorated by $\kurv$ and the other leaves decorated by $\alpha$.
    Now, the image $d_\rmT(t_1)$ has one term which is the tree $t_1'$ with leaves $\ell_1$ and $\ell_2$ decorated by $\kurv$, and the other leaves decorated by~$\alpha$.
    Since $d_\rmT(\lambda)=0$, the sum $\lambda$ must also contain the tree $t_2$, with the leaf $\ell_2$ decorated by~$\kurv$ and the other leaves decorated by $\alpha$.
    Continuing in this fashion, we obtain that $\lambda$ must contain $d_\rmT(t^\alpha)$, where $t^\alpha$ is the tree $t$ with all the leaves $\ell_0,\ldots, \ell_k$ decorated by~$\alpha$.
    Repeating the process for every distinct tree~$t$ in~$\lambda$, and using again the freeness of~$\alpha$ and~$\kurv$, we obtain the desired element $\rho \eqdef \sum_{t \in \lambda} t^\alpha$.
    The uniqueness of $\rho$ follows from the weight zero result.
\end{proof}

\begin{remark}
\label{rmk:model-category-argument}
     There is a more conceptual though less explicit proof of \cref{dT homology}.
     Since $\rmT$ is generated by elements in arity $0$, the underlying 
     module of $\rmP \, \hat{\vee} \, \rmT$ is the composite product $\rmP \circ \rmT$ (see \cite[Section 6.2.1]{LodayVallette12}).
     Using the operadic K\"unneth formula \cite[Proposition 6.2.3]{LodayVallette12}, we get that the underlying module of $H_*(\rmP \, \hat{\vee} \, \rmT, d_\rmT)$ is $H_*(\rmP \circ \rmT, d_\rmT) \cong H_*(\rmP, 0) \circ H_*(\rmT, d_\rmT)$. 
     Since $(\rmT, d_\rmT)$ is acyclic, the latter is equal to $\rmP$.
     The result of \cref{dT homology} follows.
\end{remark}

Before turning to the proof of~\cref{main theorem}, we define the candidate for the right adjoint of the forgetful functor
\begin{equation*}
    \begin{matrix}
        \calL & : & \Curv & \to & (\cAinf \downarrow \dgOp) \\
        & & \calQ := (\rmQ, d_\rmQ, \mu_0, \rmQ_0) & \mapsto & (\rmQ, d_\rmQ).
    \end{matrix}
\end{equation*}
Given a morphism $((\cAinf, d_{\cAinf}) \to (\rmP, d_\rmP))$ in the category $(\cAinf \downarrow \dgOp)$ and $n \geq 0$, we let 
    \[\mu_n^\alpha \eqdef \sum_{r_0,\ldots,r_n \geq 0} \mu_{n+r_0+\cdots+r_n} (\alpha^{r_0},-,\alpha^{r_1},-,\ldots,\alpha^{r_n}) \in \rmP \, \hat{\vee} \, [\alpha]. \]

\begin{def-prop}
\label{def-prop:adjoint}
    There is a functor $\calR : \cMult \to \Curv$ acting on objects as $((\cAinf, d_{\cAinf}) \to (\rmP, d_\rmP)) \mapsto (\rmP \, \hat{\vee} \, \rmT, d, \kurv + \mu_0^\alpha, \rmP \, \hat{\vee} \, [\alpha])$, where
    \[d_{|\rmP} = d_\rmP, \quad d \alpha = \kurv = -\mu_0^{\alpha} + (\kurv + \mu_0^{\alpha}), \quad d \kurv = 0. \]
    and acting on morphisms as $f \mapsto f \vee 1_{\rmT}$.
    The unique morphism $c \calA_{\infty} \to \calR(\rmP, d_\rmP)$ in $\Curv$ sends $\mu_n$ to $\mu_n^{\alpha}$ for $n \geq 1$.
\end{def-prop} 
\begin{proof}
    We need to check that the image under $\calR$ of an object is a well defined object in $\Curv$.
    The only non-trivial thing to check is condition $(\ref{d1 closed})$ in \cref{curv}.
    This holds since $d_1(\kurv + \mu_0^{\alpha}) = d_1^2 \alpha = 0$.

    It remains to show that the assignment $\mu_0 \mapsto (\kurv + \mu_0^{\alpha})$ and $\mu_n \mapsto \mu_n^{\alpha}$ for $n \geq 1$ defines a morphism $c \calA_{\infty} \to \calR(\rmP, d_\rmP)$ in $\Curv$.
    To see that the latter is a dg morphism, observe that it is the composition of the map $\cAinf \to \cAinf \, \hat{\vee} \, \rmT$ of \cref{prop:twist-morphism-ns}, and the map $\cAinf \, \hat{\vee} \, \rmT \to \rmP \, \hat{\vee} \, \rmT$ induced by the structural morphism $\cAinf \to \rmP$.
    Once this property is established, it is easy to see that it defines a morphism in $\Curv$.
\end{proof}

We now turn to the proof of the theorem. 

\begin{proof}[Proof of \cref{main theorem}]
Given $f: (\rmQ, d_\rmQ) \to (\rmP, d_\rmP)$ over $\cAinf$, we want to prove that there exists a unique morphism $\Phi \in \hom_{\Curv}(\calQ, \calR(\rmP,d_\rmP))$ such that $\epsilon_{\calQ} \circ \calL \Phi = f$.

\noindent
We will use the additional complete filtration on $\rmP \, \hat{\vee} \, \rmT$ by the number of $\alpha$ appearing.  
Given a $\Phi : \rmQ \to \rmP \, \hat{\vee} \, \rmT$, we split $\Phi = \sum \Phi^k$ for the decomposition into homogenous pieces in~$\alpha$. 

We start with uniqueness.
Suppose given two maps with the desired property, i.e. some 
$\Phi_-, \Phi_+ \in \hom_{\Curv}(\calQ, \calR(\rmP, d_\rmP))$ such that $\epsilon_{\calQ} \circ \calL \Phi_\pm = f$.
Since $\rmQ_0 \vee [\mu_0] \xrightarrow{\sim} \rmQ$
and any morphism in $\Curv$ has fixed behaviour on $\mu_0$, 
it is enough to check 
$\Phi_+ = \Phi_-$ on elements of $\rmQ_0$. 
Now for $\nu \in \rmQ_0$, $\Phi_{\pm}(\nu) \in \rmP \, \hat{\vee} \, [\alpha]$; since $\epsilon_{\calQ} \circ \calL \Phi_\pm = f$, we must have $\Phi_{\pm}^0(\nu) = f(\nu)$. 

As $\Phi_{\pm}$ are dg operad morphisms and only $d_\rmT$ affects the number of $\alpha$, we have: 
\[d_{\rmT} \circ \Phi^{k+1}_\pm  = \Phi^k_{\pm} \circ d_\rmQ - d_\rmP \circ \Phi^k_{\pm}. \]
Suppose inductively that  $\Phi^k_+ = \Phi^k_-$.  Then for $\nu \in \rmQ_0$, we see that $(\Phi^{k+1}_+(\nu) - \Phi^{k+1}_-(\nu))$ is a weight zero element in the kernel of  $d_\rmT$. 
According to Lemma \ref{dT homology}, we get $\Phi^{k+1}_+(\nu) = \Phi^{k+1}_-(\nu)$.
This concludes the proof of uniqueness.

We now prove the existence
by constructing, inductively in $k$, the desired map $\Phi = \sum \Phi^k$.  
Since the behaviour of $\Phi$ is fixed on $\mu_0$, we have to construct $\Phi$ on $\rmQ_0$.
We set $\Phi^0|_{\rmQ_0}:=  f|_{\rmQ_0}$. 

The following additional properties will ensure $\Phi := \sum_{j \geq 0} \Phi^j : \calQ \to \calR(\rmP)$
is a morphism such that $\epsilon_{\calQ} \circ \calL \Phi = f$: 
\begin{enumerate}
    \item\label{item filtration property} 
    For every $j \in \{0, \dots, k \}$ and $\nu \in \rmQ_0$, we have $\Phi^j(\nu) \in \rmP \, \hat{\vee} \, [\alpha]$.
    \item\label{item operad property} 
    For every $j \in \{0, \dots, k \}$ and $\nu_1, \nu_2 \in \rmQ_0$, we have
    \[\Phi^j(\nu_1 \circ_i \nu_2) = \sum_{\substack{p, q \geq 0 \\ p+q=j}} \Phi^p(\nu_1) \circ_i \Phi^q(\nu_2). \]
    \item\label{item dg property} 
    For every $j \in \{0, \dots, k-1\}$, we have on $\rmQ_0$
    \[d_\rmT \circ \Phi^{j+1} = \Phi^j \circ d_\rmQ - d_\rmP \circ \Phi^j. \]
    \item\label{item expression on mu1} For every $j \in \{0, \dots, k \}$ and $n \geq 1$, 
    \[\Phi^j(\mu_n) = \mu_n^{\alpha,j} \eqdef \sum_{\substack{r_0, \dots, r_n \geq 0 \\ r_0 + \dots + r_n = j}} \mu_{n+j} (\alpha^{r_0},-,\alpha^{r_1},-,\ldots,\alpha^{r_n}). \] 
\end{enumerate}
    
Assume now that, given $k \in \Z_{\geq 0}$, we defined $\Phi^0, \Phi^1, \ldots, \Phi^k$ satisfying the properties above.
Given $\nu \in \rmQ_0$, we consider 
\[\lambda^k(\nu) := (\Phi^k \circ d_\rmQ - d_\rmP \circ \Phi^k) (\nu). \]
By construction, $\lambda^0(\nu)$ has no $\alpha$, so $d_\rmT \lambda^0(\nu)=0$. If $k \geq 1$: 
\begin{eqnarray*}
    d_\rmT(\lambda^k(\nu)) &=& (d_\rmT \circ \Phi^k \circ d_\rmQ - d_\rmT \circ d_\rmP \circ \Phi^k) (\nu) \\
    &=& (d_\rmT \circ \Phi^k \circ d_\rmQ + d_\rmP \circ d_\rmT \circ \Phi^k) (\nu) \\
    &=&  ((\Phi^{k-1} \circ d_\rmQ - d_\rmP \circ \Phi^{k-1})\circ d_\rmQ + d_\rmP \circ (\Phi^{k-1} \circ d_\rmQ - d_\rmP \circ \Phi^{k-1})) (\nu) 
    = 0.
\end{eqnarray*} 
By construction, $\lambda^0(\nu) = (\Phi^0-f)(d_1 \nu)$ has $\kurv$-weight one.
If $k\geq 1$, then by 
\eqref{item filtration property}, we have that $\Phi^k(\nu)$ and hence $d_\rmP \Phi^k(\nu)$ has $\kurv$-weight zero, while $\Phi^k(d_\rmQ \nu)$ has $\kurv$-weight $\le 1$. 
So $\lambda^k(\nu)$ is an element of positive $\alpha$-filtration and $\kurv$-weight $\le 1$ such that $d_\rmT \lambda^k(\nu) = 0$.
According to \cref{dT homology} there exists a unique $\rho \in \rmP \, \hat{\vee} \, [\alpha]$ such that $\lambda = d_\rmT(\rho)$.
We define $\Phi^{k+1}(\nu) := \rho$.

We have now defined $\Phi^{k+1}$ on $\rmQ_0$.
By construction, it satisfies $(\ref{item filtration property})$ and $(\ref{item dg property})$. 
Using the assumptions satisfied by $(\Phi^j)_{0 \leq j \leq k}$, it is straightforward to show that, for every $\nu_1, \nu_2 \in \rmQ_0$, the sum
\[\sum_{\substack{p, q \geq 0 \\ p+q=k+1}} \Phi^p(\nu_1) \circ_i \Phi^q(\nu_2) \in \rmP \, \hat{\vee} \, [\alpha] \]
is mapped by $d_\rmT$ to $(\Phi^k \circ d_\rmQ - d_\rmP \circ \Phi^k) (\nu_1 \circ_i \nu_2)$.
But we have already seen
(from \cref{dT homology}) 
that $\Phi^{k+1}(\nu)$ is the unique element with this property. This establishes $(\ref{item operad property})$.

In order to see that $(\ref{item expression on mu1})$ holds, it is enough to check the computation
\begin{align*}
    \Phi^k(d_\rmQ \mu_n) - d_\rmP(\Phi^k \mu_n) & = \Phi^k\left(- \sum_{\substack{p,q \geq 0 \\ p+q=n}} \mu_{p+1} \star \mu_q \right) - d_\rmP (\mu_n^{\alpha,k}) \\
    & = - \sum_{\substack{i, j \geq 0 \\ i+j=k}} \sum_{\substack{p,q \geq 0 \\ p+q=n}} \Phi^i(\mu_{p+1}) \star \Phi^j(\mu_q) - d_\rmP (\mu_n^{\alpha,k}) \\
    & = - \sum_{\substack{i, j \geq 0 \\ i+j=k}} \sum_{\substack{p,q \geq 0 \\ p+q=n}} \mu_{p+1}^{\alpha,i} \star \mu_q^{\alpha,j} - \mu_{n+1}^{\alpha,k} \star \kurv  - d_\rmP (\mu_n^{\alpha,k}) \\
    & = - \mu_{n+1}^{\alpha,k} \star \kurv = d_\rmT (\mu_n^{\alpha,k+1}).
\end{align*}
This finishes the proof.
\end{proof}


\section{The symmetric case}
\label{section: symmetric}

Here we prove Theorems \ref{thm:morphism-sym} and \ref{main theorem symmetric}, which are the versions of our main results for symmetric operads.  The sole difference in the proofs concern the lemmas invoked in the arguments; the remainder, which we omit, is formally identical: one literally has to replace ``non-symmetric'' by ``symmetric'' and the operad $\cAinf$ by the operad $\cLinf$.

\subsection{Conventions}

From here on, the word ``operad'' will mean symmetric operad in graded vector spaces over a field of characteristic zero $\K$, and the expression ``dg operad'' will mean symmetric operad in differential graded vector spaces over $\K$.
In the latter case, we will use homological degree conventions. 
We refer to \cite{LodayVallette12} for a more details on algebraic operads.

\subsection{Basic definitions} 

Let $\mathbb{S}_n$ denote the symmetric group of degree~$n$.
Recall that a dg $\mathbb{S}$-module is a family of dg vector spaces $\{ \rmP(n) \}_{n \geq 0}$ endowed with an action of $\mathbb{S}_n$ for each $n$. 
Given a a graded dg $\mathbb{S}$-module $(A, d_A)$, the symmetric group action on the endomorphism dg operad $\End_{(A,d_A)}$ is given by permuting the factors in $A^{\otimes n}$.

\begin{definition}
    The shifted \defn{L-infinity dg operad}, denoted $\Linf$, is the free operad on generators~$\ell_n$, $n\geq 2$, of arity $n$ and degree $(-1)$, endowed with the differential
    \begin{equation} \label{linf differential}
        d_{\Linf}(\mu_n) \eqdef -\sum_{\substack{p \geq 1, q \geq 2 \\ p+q=n}} \sum_{\sigma \in \Sh_{p,q}^{-1}}(\ell_{p+1} \circ_{1} \ell_q)^\sigma \ .
    \end{equation}
\end{definition}

\begin{definition}
    The shifted \defn{augmented L-infinity dg operad}, denoted $\dLinf$, is the free operad on generators~$\ell_n$, $n\geq 1$, of arity $n$ and degree $(-1)$, endowed with the differential
    \begin{equation} \label{dlinf differential}
        d_{\dLinf}(\mu_n) \eqdef -\sum_{\substack{p \geq 0, q \geq 1 \\ p+q=n}} \sum_{\sigma \in \Sh_{p,q}^{-1}}(\ell_{p+1} \circ_{1} \ell_q)^\sigma \ .
    \end{equation}
\end{definition}

\begin{definition}
    The shifted \defn{curved L-infinity dg operad}, denoted $\cLinf$, is the free operad on generators $\ell_n$, $n\geq 0$, of arity $n$ and degree $(-1)$, endowed with the differential
    \begin{equation}\label{clinf differential}
        d_{\cLinf}(\ell_n) \eqdef -\sum_{\substack{p,q \geq 0 \\ p+q=n}} \sum_{\sigma \in \Sh_{p,q}^{-1}}(\ell_{p+1} \circ_{1} \ell_q)^\sigma \ .
    \end{equation}
    The symmetric group action on the generators is given by $\ell_n^\sigma=\ell_n$ for any $\sigma \in \mathbb{S}_n$.
\end{definition}

For a symmetric operad $\rmP$, the \defn{pre-Lie product} of two elements $\mu \in \mathrm{P}(m)$ and $\nu \in \mathrm{P}(n)$ is defined by
\begin{equation}
\label{pre-Lie-prod-sym}
    \mu \star \nu \eqdef \sum_{\sigma \in \Sh_{m-1,n}^{-1}}(\mu \circ_1 \nu)^{\sigma}.
\end{equation}
It is usually defined on the totalization $\Tot(\rmP) \eqdef \oplus_n \rmP(n)^{\Sym_n}$ of $\rmP$, and it satisfies the pre-Lie relation~\eqref{eq:pre-Lie}.
Note that if $\nu$ has arity zero, the pre-Lie product then reduces to $\mu \star \nu = \mu \circ_1 \nu$.
One can for instance rewrite the $\cLinf$ relations~\eqref{clinf differential} as
\[ 
d_{\cLinf}(\ell_n) = - 
\sum_{\substack{p,q \geqslant 0 \\ p+q=n}}
 \ell_{p+1} \star \ell_q .
 \]


\subsection{Proof of \cref{thm:morphism-sym}}

\begin{lemma} 
\label{bracket acyclicity symmetric case} 
    Let $\kurv$ be an arity zero operation of odd degree, and $\mathrm{P} = \mathrm{P}_0 \vee [\kurv]$. 
    Then $p \mapsto p \star \kurv$ is a differential on the totalization $\Tot(\rmP)$, whose only homology consists of arity zero elements. 

    In particular, if $p \in \rmP(n)^{\Sym_n}$ is of arity $n \geq 1$, of $\kappa$-weight $1$, and satisfies $p \star \kurv=0$, then there exists a unique $p_0 \in \rmP_0(n+1)^{\Sym_{n+1}}$ such that $p = p_0 \star \kappa$.
\end{lemma}

\begin{proof}
    The fact that $p \mapsto p \star \kurv$ is a differential on the totalization of $\rmP$ follows from the graded associativity of the operadic composition and the fact that $p$ is $\Sym_n$-invariant.

    We now prove the statement about the homology of this differential.
    Let $p \in \mathrm{P}(n)^{\Sym_n}$ be an element of arity $n \geq 1$ such that $p \star \kurv = 0$.
    If $p$ is of weight $0$ (i.e. $p \in \mathrm{P}_0(n)^{\Sym_n}$), then $p=0$ since $\mathrm{P} = \mathrm{P}_0 \vee [\kurv]$.
    Assume that $p$ is of weight $k \geq 1$, and write $p = p_0 \star \kurv$ with $p_0 \in \mathrm{P}(n+1)$ of weight $(k-1) \geq 0$. 
    We need to show that $p_0$ is in fact invariant under the action of $\Sym_{n+1}$.
    Note that  
    $(p_0 \star \kurv) \star \kurv = p \star \kurv = 0$ implies $p_0=p_0^{(12)}$.
    Moreover, $p$ is invariant under the action of $\Sym_n$ by assumption. Therefore, the equation $p = p_0 \star \kurv$ implies that $p_0$ is invariant under the action of any $\sigma \in \Sym_{n+1}$ which fixes~$1$. 
    Combining this with the fact that $p_0$ is invariant under the transposition $(12)$, we deduce that $p_0$ is invariant under the action of the whole group $\Sym_{n+1}$.
    This proves the statement about the homology of $- \star \kurv$.

The remainder of the proof is obtained \emph{mutatis mutandis} from the proof of \cref{bracket acyclicity}.
\end{proof}

The remainder of the proof of \cref{thm:morphism-sym} is obtained \emph{mutatis mutandis} from the proof of \cref{thm:morphism}, using \cref{bracket acyclicity symmetric case} in place of \cref{bracket acyclicity}, and the symmetric pre-Lie product~\eqref{pre-Lie-prod-sym} in place of the non-symmetric one~\eqref{pre-Lie-prod}.


\subsection{Proof of \cref{main theorem symmetric}}

The essential ingredient needed for the symmetric case is the symmetric version of \cref{dT homology}.

\begin{lemma}
\label{dT homology symmetric}
    The $\kurv$-weight zero homology (= kernel) of $d_\rmT$ is $\rmP \vee 0 \subset \rmP \, \hat{\vee} \, \rmT$. 
    The $\kurv$-weight one homology vanishes.
    
    Moreover, if $\lambda = \mu(\alpha^k, \kurv, -)$ is $d_\rmT$-closed with $\mu \in \rmP$ and $k\geq0$, then 
    \[\lambda = d_{\rmT} \left( \frac{(-1)^{|\mu|}}{k+1} \mu(\alpha^{k+1}, -) \right). \]
\end{lemma}
\begin{proof}
    We start with an observation: given an operation $\nu \in \calP$ of arity $n$ and $0 \leq i \leq j \leq n$, we have $\mu^{\sigma}(\alpha^j,-) = \mu(\alpha^j,-)$ for every $\sigma$ in $\Sym_i$ (seen as a subgroup of $\Sym_n$) since $\alpha$ is of even degree.

    Weight zero: let $\lambda$ in positive $\alpha$-filtration and of $\kurv$-weight zero such that $d_{\rmT} (\lambda) = 0$.
    Now let $\mu \in \calP$ such that $\lambda = \mu(\alpha^{k+1}, -)$.
    Let $\nu := \frac{1}{(k+1)!} \sum_{\sigma \in S_{k+1}} \mu^\sigma$, so that $\nu^{\sigma} = \nu$ for every $\sigma \in \Sym_{k+1}$.
    According to the observation in the beginning of the proof, we have $\nu(\alpha^{k+1}, -) = \mu(\alpha^{k+1}, -)$.
    In particular, we have $\lambda = \nu(\alpha^{k+1}, -)$.
    Now we have
    \begin{align*}
        d_{\rmT} (\lambda)  = d_{\rmT} (\nu(\alpha^{k+1}, -)) 
        & = (-1)^{|\nu|} \sum_{i=1}^{k+1} \nu(\alpha^{i-1}, \kurv, \alpha^{k+1-i}, -) \\ 
        & = (-1)^{|\nu|} \sum_{i=1}^{k+1} \nu^{(i,k+1)}(\alpha^{k}, \kurv, -) \\
        & = (-1)^{|\nu|} (k+1) \, \nu(\alpha^{k}, \kurv, -).
    \end{align*}
    Therefore the assumption $d_{\rmT} (\lambda) = 0$ implies $\nu(\alpha^{k}, \kurv, -) = 0$.
    This implies that $\nu(\alpha^k,-) = 0$, and therefore $\lambda = \nu(\alpha^{k+1}, -) = 0$.

    Weight one: let $\lambda$ of $\kurv$-weight one such that $d_{\rmT} (\lambda) = 0$.
    Now let $\mu \in \calP$ such that $\lambda = \mu(\alpha^k, \kurv, -)$. 
    Let $\nu := \frac{1}{k!} \sum_{\sigma \in \Sym_k} \mu^\sigma$, so that $\nu^{\sigma} = \nu$ for every $\sigma \in \Sym_k$.
    According to the observation in the first paragraph of the proof, we have $\nu(\alpha^k, -) = \mu(\alpha^k, -)$ and $\nu(\alpha^{k+1}, -) = \mu(\alpha^{k+1}, -)$.
    In particular, we have $\lambda = \nu(\alpha^k, \kurv, -)$.
    Now given $i_0 \in \{1, \dots, k\}$, we have
    \begin{align*}
        d_{\rmT} (\lambda)  = d_{\rmT} (\nu(\alpha^k, \kurv, -)) 
        & = (-1)^{|\nu|} \sum_{i=1}^k \nu(\alpha^{i-1}, \kurv, \alpha^{k-i}, \kurv, -) \\ 
        & = (-1)^{|\nu|} \sum_{i=1}^k \nu^{(i,i_0)}(\alpha^{i_0-1}, \kurv, \alpha^{k-i_0}, \kurv, -) \\
        & = (-1)^{|\nu|} k \, \nu(\alpha^{i_0-1}, \kurv, \alpha^{k-i_0}, \kurv, -).
    \end{align*}
    Therefore the assumption $d_{\rmT} (\lambda) = 0$ implies $\nu(\alpha^{i-1}, \kurv, \alpha^{k-i}, \kurv, -) = 0$ for every $i$ in $\{1, \dots, k\}$.
    Since $\kurv$ is of odd degree, this implies that 
    \[\nu^{(i, k+1)}(\alpha^{i-1}, -, \alpha^{k-i}, -, -) = \nu(\alpha^{i-1}, -, \alpha^{k-i}, -, -) \]
    for every $i$ in $\{1, \dots, k\}$.
    Now we compute 
    \begin{align*}
        d_{\rmT} \left( \frac{(-1)^{|\mu|}}{k+1} \mu(\alpha^{k+1}, -) \right)  = \frac{(-1)^{|\nu|}}{k+1} d_{\rmT} (\nu(\alpha^{k+1}, -)) 
        & = \frac{1}{k+1} \sum_{i=1}^{k+1} \nu(\alpha^{i-1}, \kurv, \alpha^{k+1-i}, -) \\ 
        & = \frac{1}{k+1} \sum_{i=1}^{k+1} \nu^{(i,k+1)}(\alpha^k, \kurv, -) \\
        & = \nu(\alpha^k, \kurv, -) = \lambda.
    \end{align*}
    This finishes the proof.
\end{proof}

\begin{remark}
    The argument given in \cref{rmk:model-category-argument} applied in the symmetric case gives a more conceptual proof of the first part of \cref{dT homology symmetric}.
\end{remark}

The remainder of the proof of \cref{main theorem symmetric} works \emph{mutatis mutandis} as the one of \cref{main theorem}, modulo the addition of the condition that $\Phi$ is equivariant with respect to the symmetric groups action. 


\appendix 
\section{Twisting morphisms for $\cAinf$ and $\cLinf$}
\label{sec:appendix}

\begin{prop}
\label{prop:twist-morphism-ns}
    The formula \eqref{Maurer-Cartan map} defining $\tw_{\cAinf}$ gives a morphism of non-symmetric dg operads.
\end{prop}

\begin{proof}
    Let us abbreviate $\tw_{\cAinf}$ by $\tw$. 
    We compute
    \begin{align*}
    d_{\cAinf \, \hat{\vee} \, \rmT}(\tw(\mu_n)) 
        &= 
        \sum_{i_0,\ldots,i_n} d_{\cAinf}(\mu_{i_0+\cdots + i_n+n})(\alpha^{i_0}, - , \alpha^{i_1}, - , \cdots, - , \alpha^{i_{n}}) \\
        &- \sum_{i_0,\ldots,i_n} \sum_{k=0}^{n}\sum_{i=1}^{i_k} \mu_{i_0+\cdots+i_n+n}(\alpha^{i_0},-,\cdots,\alpha^{i-1},\kurv,\alpha^{i_k-i},-,\cdots,\alpha^{i_n}) .        
    \end{align*}
    We treat the two sums on the right hand side separately.
      \begin{align*}
        & \sum_{i_0,\ldots,i_n} d_{\cAinf}(\mu_{i_0+\cdots + i_n+n})(\alpha^{i_0}, - , \alpha^{i_1}, - , \cdots, - , \alpha^{i_{n}}) \\
        & =  - \sum_{i_0,\ldots,i_n} 
        \sum_{\substack{p+q+r=\\ \sum i_k+n}}
        (\mu_{p+1+r} \circ_{p+1} \mu_q)(\alpha^{i_0}, - , \alpha^{i_1}, - , \cdots, - , \alpha^{i_{n}}) \\
        & =  - \sum_{\substack{i_0,\ldots,i_a,j_0,\ldots,j_b \\ a+b=n, a \geqslant 1}} 
        \sum_{\substack{p+q+r= \\ \sum i_k + \sum j_l+n}}
        \mu_{p+1+r}(\alpha^{i_0}, - , \cdots, - , \alpha^{i_{a}}) \circ_{p+1} \mu_q(\alpha^{j_0}, - , \cdots, - , \alpha^{j_{b}}) \\
        & =  -  
        \sum_{\substack{a+b=n\\ 1\leqslant c \leqslant a}}
        \left(
        \left(
        \sum_{i_0,\ldots,i_a}\mu_{\sum i_k+a}(\alpha^{i_0}, - , \cdots, - , \alpha^{i_{a}}) \right) \circ_{c} 
        \left(\sum_{j_0,\ldots,j_b}\mu_{\sum j_l+b}(\alpha^{j_0}, - , \cdots, - , \alpha^{j_{b}})\right)\right) \\
        & =  -  
        \sum_{p+q+r=n}
        \left(
        \left(
        \sum_{i_0,\ldots,i_{p+1+r}}\mu_{\sum i_k+p+1+r}(\alpha^{i_0}, - , \cdots, - , \alpha^{i_{p+1+r}}) \right) \circ_{p+1} 
        \left(\sum_{j_0,\ldots,j_q}\mu_{\sum j_l+q}(\alpha^{j_0}, - , \cdots, - , \alpha^{j_{q}})\right)\right)
    \end{align*}
    On the other side, we have
    \begin{eqnarray*}
        \sum_{i_0,\ldots,i_n} \sum_{k=0}^{n}\sum_{i=1}^{i_k} \mu_{i_0+\cdots+i_n+n}(\alpha^{i_0},-,\cdots,\alpha^{i-1},\kurv,\alpha^{i_k-i},-,\cdots,\alpha^{i_n}) 
         =  \sum_{0 \leqslant p \leqslant n} \tw(\mu_{n+1}) \circ_{p+1} \kurv.
    \end{eqnarray*}
    Combining the previous two computations, we get
    \begin{align*}
        d_{\cAinf \, \hat{\vee} \, \rmT}(\tw(\ell_n)) 
        =  -  
        \sum_{p+q+r=n}
        \tw(\mu_{p+1+r})\circ_{p+1} \tw(\ell_q)
        = \tw (d_{\cAinf}(\mu_n)),
    \end{align*}
    which finishes the proof.
\end{proof}

\begin{prop}
\label{prop:twist-morphism}
    The formula \eqref{symmetric Maurer-Cartan map} defining $\tw_{\cLinf}$ gives a morphism of symmetric dg operads.
\end{prop}

\begin{proof}
    Let us abbreviate $\tw_{\cLinf}$ by $\tw$. 
    We compute
    \[
        d_{\cLinf \, \hat{\vee} \, \rmT}(\tw(\ell_n)) 
        = \sum_{k \geqslant 0}\frac{1}{k!} d_{\cLinf}(\ell_{k+n})(\alpha^k,-) - \sum_{k \geqslant 1} \frac{1}{k!} \sum_{i=1}^{k} \ell_{k+n}(\alpha^{i-1},\kurv,\alpha^{k-i},-) .
    \]
    We treat the two sums on the right hand side separately. 
    On the one side, we have 
    \begin{eqnarray*}
        \sum_{k \geqslant 0} \frac{1}{k!} d_{\cLinf}(\ell_{k+n})(\alpha^k,-) & = & - \sum_{k \geqslant 0} \frac{1}{k!} \sum_{\substack{p+q=k+n \\ q \geqslant 0}}\sum_{\sigma \in \Sh_{p,q}^{-1}}(\ell_{p+1} \circ_1 \ell_q)^{\sigma}(\alpha^k,-) \\
        & = & - \sum_{k \geqslant 0}\frac{1}{k!} 
        \sum_{p+q=k+n}
        \sum_{\substack{h+j = k \\ h \leqslant p, j\leqslant q}}\binom{k}{j}
        \sum_{\tau \in \Sh_{p-h,q-j}^{-1}}[(\ell_{p+1} \circ_{1} \ell_q)(\alpha^j,-,\alpha^h,-)]^{\tau} \\
        & = & - \sum_{k \geqslant 0}
        \sum_{p+q=k+n}
        \sum_{\substack{h+j = k \\ h \leqslant p, j\leqslant q}}\frac{1}{k!} \binom{k}{j}
        \sum_{\tau \in \Sh_{p-h,q-j}^{-1}}[(\ell_{p+1}^{(1 \ h+1)} \circ_{1} \ell_q)(\alpha^j,-,\alpha^h,-)]^{\tau} \\
        & = & - \sum_{k \geqslant 0}\sum_{h+j=k}\frac{1}{h!j!} 
    \sum_{\substack{p+q=k+n \\ p\geqslant h, q  \geqslant j}}
        \sum_{\tau \in \Sh_{p-h,q-j}^{-1}}[(\ell_{p+1} \circ_{h+1} \ell_q)(\alpha^k,-)]^{\tau} \\
        & = & - \sum_{h,j\geqslant 0}\frac{1}{h!j!} 
        \sum_{\substack{p-h+q-j=n \\ p-h, q-j \geqslant 0}}
        \sum_{\tau \in \Sh_{p-h,q-j}^{-1}}[\ell_{p+1}(\alpha^h,-) \circ_{1} \ell_q(\alpha^j,-)]^{\tau} \\
        & = & -  
        \sum_{a+b=n}
        \sum_{\tau \in \Sh_{a,b}^{-1}}
        \left(
        \left(
        \sum_{h\geqslant 0}\frac{1}{h!}\ell_{h+a+1}(\alpha^h,-) \right) \circ_{1} 
        \left(\sum_{j\geqslant 0}\frac{1}{j!}\ell_{j+b}(\alpha^j,-)\right)\right)^{\tau} 
    \end{eqnarray*}
    Here, we have used that the number of $(h,j)$-unshuffles of $h+j=k$ entries is given by~$\binom{k}{j}$.
    On the other side, we have
    \begin{eqnarray*}
        \sum_{k \geqslant 1} \frac{1}{k!} \sum_{i=1}^{k} \ell_{k+n}(\alpha^{i-1},\kurv,\alpha^{k-i},-) & = & \sum_{k \geqslant 1} \frac{1}{(k-1)!} \ell_{k+n}(\alpha^{k-1},\kurv,-) \\
        & = & \sum_{k \geqslant 0} \frac{1}{k!} \ell_{k+1+n}(\alpha^k,\kurv,-) \\
        & = & \tw(\ell_{n+1}) \circ_1 \kurv.
    \end{eqnarray*}
    Combining the previous two computations, we get
    \begin{align*}
        &d_{\cLinf \, \hat{\vee} \, \rmT}(\tw(\ell_n)) \\
        & =  \sum_{k \geqslant 0}\frac{1}{k!} d_{\cLinf}(\ell_{k+n})(\alpha^k,-) - \sum_{k \geqslant 1} \sum_{i=1}^{k} \frac{1}{k!} \ell_{k+n}(\alpha^{i-1},\kurv,\alpha^{k-i},-) \\
         & =  -  
        \sum_{a+b=n}
        \sum_{\tau \in \Sh_{a,b}^{-1}}
        \left(
        \left(
        \sum_{h\geqslant 0}\frac{1}{h!}\ell_{h+a+1}(\alpha^h,-) \right) \circ_{1} 
        \left(\sum_{j\geqslant 0}\frac{1}{j!}\ell_{j+b}(\alpha^j,-)\right)\right)^{\tau} - \tw(\ell_{n+1}) \circ_1 \kurv\\
        & =  -  
        \sum_{a+b=n}
        \sum_{\tau \in \Sh_{a,b}^{-1}}(\tw(\ell_{a+1})\circ_1 \tw(\ell_b))^\tau \\
        &= \tw (d_{\cLinf}(\ell_n)),
    \end{align*}
    which finishes the proof.
\end{proof}


\bibliography{uncurving}

\end{document}